\documentclass[12pt]{amsart}

\usepackage{amsmath, amstext, amsgen, amsbsy, amsopn}
\usepackage{amsfonts, amssymb}
\usepackage{graphicx}

\usepackage{pdfsync}

\newtheorem{theorem}{Theorem}

\newtheorem{lemma}[theorem]{Lemma}
\newtheorem{corollary}[theorem]{Corollary}

\title[Triple-Crossing Diagrams]{Triple-Crossing Number and Moves on Triple-Crossing Link Diagrams}
\author[Adams, Hoste, and Palmer]{
Colin Adams\\
Williams College\\
\\
Jim Hoste\\
Pitzer College\\
\\
Martin Palmer\\
Universit\"at Bonn
}
\thanks{The third author was partially supported by project KH3CF of the Initiative d'excellence at Universit\'e Sorbonne Paris Cit\'e (and also partially by UniversitŽ Paris 13 after the suspension of the IDEX at USPC) and Universit\"at Bonn.}

\begin{document}
\nocite{Adams1, Adams2, Adams3, Adams4, Adams5}
\maketitle
\begin{abstract} Every link in the 3-sphere has a projection to the plane where the only singularities are pairwise transverse triple points. The associated diagram, with height information at each triple point, is a {\it triple-crossing diagram} of the link. We give a set of diagrammatic moves on triple-crossing diagrams analogous to the Reidemeister moves on ordinary diagrams. The existence of $n$-crossing diagrams for every $n>1$ allows the definition of the {\it $n$-crossing number}. We prove that for any nontrivial, nonsplit link, other than the Hopf link, its triple-crossing number is strictly greater than its quintuple-crossing number.
\end{abstract}

\section{Introduction} The  classical theory of  knots and links is often approached via link {\it diagrams} and the well-known {\it Reidemeister moves}. A link diagram is a projection of the link to a 2-sphere, which we think of as a plane union a point at infinity,  having only transverse double points as singularities and with an indication at each double point of which strand is ``on top.'' This is done by erasing a small section of the under-crossing strand. The Reidemeister moves consist of local transformations which can be used to  alter a diagram while not changing the underlying link. In general, we will refer to any such diagrammatic change as a diagram {\it move}. The important result, proven by both Reidemeister \cite{Reidemeister} and Alexander and Briggs \cite{AlexanderBriggs}, is that two  diagrams represent the same  link if and only if they are related by a sequence of Reidemeister moves. Of course, we also allow any ambient isotopy of the diagram in the projection 2-sphere.  Thus we may treat the objects of knot theory as equivalence classes of diagrams.

When working with oriented links, there is a  corresponding set of oriented Reidemeister moves. In \cite{Polyak} it is shown that a set of four oriented Reidemeister moves, called $\Omega 1a$, $\Omega 1b$, $\Omega 2a$, and $\Omega 3a$,  and their inverses, are sufficient to pass between any two oriented diagrams of the same link. Later in this paper, we will extend these moves to four moves which we call  $C 1a$, $C 1b$, $C 2a$, and $C 3a$ pictured in Figure~\ref{CCC R moves}. If the dashed curves in the figure (which will be explained later) are ignored, then $C 1a$, $C 1b$, $C 2a$, and $C 3a$ reduce to $\Omega 1a$, $\Omega 1b$, $\Omega 2a$, and $\Omega 3a$, respectively.  If we ignore the orientations in $\Omega 1a$, $\Omega 1b$, $\Omega 2a$, and $\Omega 3a$ we obtain a set of three unoriented moves sufficient to pass between all unoriented diagrams of the same unoriented link.

Recently, several papers have explored the topic of link diagrams with {\it multicrossings}. See, for example,  \cite{Adams1}--\cite{Adams4}. In these diagrams, $n$ strands are allowed to cross at a single point in the plane (still pairwise transversely), creating what is known as an {\it $n$-crossing}. Now each $n$-crossing must be accompanied with a labeling of the strands, $1, 2, \dots, n$,  from top to bottom in order to depict the link in space. Many of the obvious results analogous to classical diagrams have been proven. 
For example, given any $n>1$, every link has an {\it $n$-diagram}, that is, one with only $n$-crossings. However, until now, no analog of the Reidemeister moves have been found for multicrossing diagrams. In Sections~\ref{moves on 3-diagrams}--\ref{proof of main theorem} we describe a set of moves on 3-diagrams and prove that they are sufficient to pass between all 3-diagrams of the same link, as long as the ``natural'' orientation of the 3-diagrams define the same oriented link, up to a certain equivalence. 

Because every link has an $n$-diagram for every $n$,  the {\it  $n$-crossing number}, $c_n(L)$,  of a link $L$ may be defined as the smallest number of $n$-crossings in any $n$-diagram of $L$.  See \cite{Adams1}. It is known that $c_2(L)>c_3(L)$ and $c_2(L)> c_4(L)$. Moreover, it is known that $c_n(L) \geq c_{n+2}(L)$ and $c_n(L) \geq c_{2n}(L)$ for all $n \geq 2$ and all $L$. See \cite{Adams5}. In Section~\ref{c_3 and c_5} of this paper we prove that for any  nontrivial, nonsplit link $L$, other than the Hopf link, $c_3(L) > c_5(L)$.

The authors thank the organizers of {\it Knots in Hellas}, a conference held in Olympia, Greece, July 17--23, 2016, where this work originated.

\section{Moves on 3-Diagrams}\label{moves on 3-diagrams}

Before describing our set of 3-diagram moves, we begin by noticing that an unoriented 3-diagram can be given a  natural orientation by using the checkerboard coloring of its complement. In fact, this is true for  all $(2n+1)$-diagrams for any $n\ge 1$, and is quite different from the case of $(2n)$-diagrams.

\begin{lemma}
The complementary regions of any $n$-diagram may be colored black and white, checkerboard fashion.
\end{lemma}
\begin{proof} By slightly perturbing the strands near each multicrossing of an $n$-diagram $D$, each $n$-crossing can be separated into 
 $n(n-1)/2$ classical crossings, creating a classical diagram $d$ of the same link. Now $d$ can be checkerboard colored and that coloring induces one on $D$ when the perturbed strands are returned to their original positions.
\end{proof}

If $D$ is a $(2n+1)$-diagram, we can use the checkerboard coloring to orient $D$  by orienting the  boundary edges of each black region counterclockwise.  Because each multicrossing involves an odd number of strands, the orientations match up to give an orientation of $D$. We call this a {\it natural} orientation of $D$. Swapping the colors reverses the orientation. Thus there are two possible natural orientations. If $D$ is disconnected, meaning that the associated projection\footnote{By a {\it projection}, we mean an actual projection with transverse double points (or multi-points), as opposed to a {\it diagram} where heights have been indicated at every multi-point.} $P$ is a disconnected subset of the projection plane, then we will refer to the diagram associated to a maximal, connected, subset of $P$ as a  {\it subdiagram} of $D$.
If $D$ is disconnected, then 
a {\it piecewise natural} orientation of $D$ is one that is natural on each subdiagram of $D$. If $D$ is the union of $j$ disjoint subdiagrams, then there are $2^j$ piecewise natural orientations of $D$.

In any 3-diagram, call a 3-crossing {\it Type A} if, as we encircle the crossing, the strands are oriented in--out--in--out--in--out. Considering Type A crossings provides an alternate approach to orienting a 3-diagram.

\begin{lemma}\label{natural iff Type A} An orientation of a 3-diagram $D$ is piecewise natural if and only if all crossings are Type A.
\end{lemma}
\begin{proof}
If the orientation of a diagram is piecewise natural, then it is easy to see that every crossing is Type A. Conversely, suppose $D$ is   oriented so that every crossing is Type A  and that $D'$ is a  subdiagram of $D$. Let $R$ be a complementary region of $D'$. Because all the crossings are of Type A, it follows that the boundary edges of $R$ are all oriented in the same direction  as we go around the boundary of $R$. Therefore, if the complementary regions of $D'$ are checkerboard colored black and white, then the orientations on the boundaries of the black regions are either all clockwise or all counterclockwise.  Thus the orientation of $D$ is piecewise natural.
\end{proof}

Figures~\ref{1and2-move},  \ref{basepointMove} and \ref{bandMove} illustrate the {\it 1-move, 2-move,  basepoint move}, and {\it band move} on 3-diagrams. The 1-move can be thought of as first performing a Type I Reidemeister move on the strand on the right and then sliding the 1-gon  over the strand on the left. We may think of the 2-move as performing a Type II Reidemeister move on the two outer strands with the center strand lying  above. 
\begin{figure}[htbp]
\vspace*{13pt}
\hskip -.5 in \includegraphics*[scale=.25]{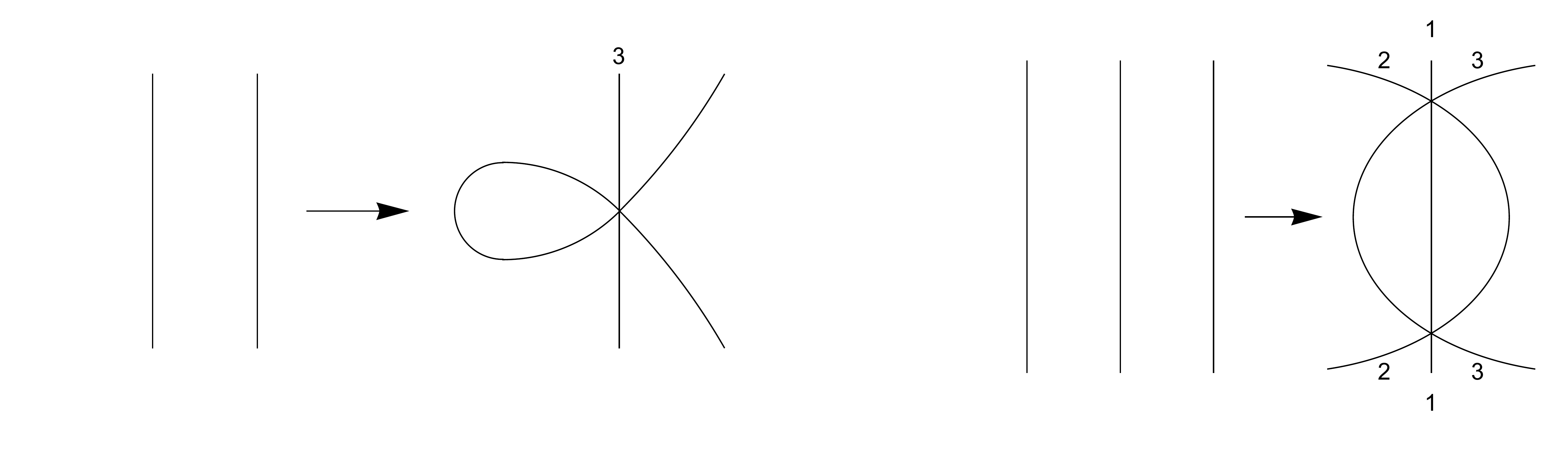}
\caption{The 1-move and 2-move.}
\label{1and2-move}
\end{figure}

A basepoint move is illustrated in Figure~\ref{basepointMove}, and is defined as follows.\footnote{Why we call this a basepoint move will become clear in Section~\ref{crossing circles}.} Suppose the 3-diagram $D$ contains a trivial 3-string tangle $T_1$ consisting of three parallel arcs, two of which are boundary parallel (the {\it outer} arcs) and one of which is not (the {\it central} arc). The  endpoints  of the three arcs come in   {\it opposite} pairs in the obvious way. Suppose that $\alpha$ is an arc of $D$ whose endpoints are a pair of opposite endpoints of $T_1$ belonging to the outer arcs. Suppose further that $\alpha$ is an {\it over-arc} of $D$, that is, for every crossing that $\alpha$ passes through, it lies on top. Moreover, assume that $\alpha$ passes through at least one 3-crossing and let $T_2$ be a small disk centered at that crossing. To make the move, we replace $T_1$ with a 3-crossing $T_1'$ where the central arc of $T_1$ becomes the lowest strand of $T_1'$, the endpoints of $\alpha$ are joined to form the highest strand, and the other pair of endpoints are joined to form the middle strand. Additionally, the crossing at $T_2$ is replaced with a trivial tangle $T_2'$ having as central strand what was the lowest strand of the crossing. Figure~\ref{basepointMove} shows how a basepoint move might appear. Notice that $\alpha$ may pass though many 3-crossings and $T_2$ does not need to be the 3-crossing nearest to $T_1$. Moreover, the heights of the strands in $T_2$ and the choice of endpoints  of $\alpha$ in $T_1$ may not appear the same as in Figure~\ref{basepointMove}. 

\begin{figure}[htbp]
\vspace*{13pt}
\centerline{\includegraphics*[scale=.4]{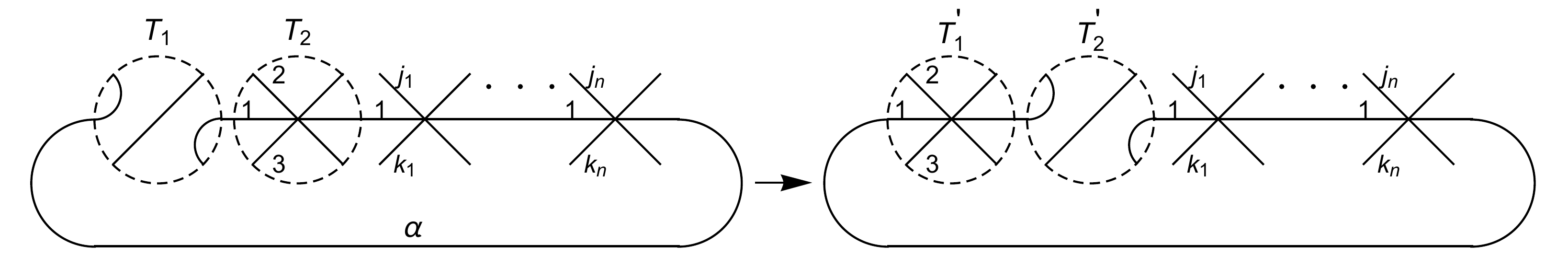}}
\caption{A possible basepoint move.}
\label{basepointMove}
\end{figure}

 To see that the basepoint move does not change the underlying link, simply pick up the arc $\alpha$ and then lay it down inside the tangle $T_1$. This creates a 2-crossing in $T_1$ and reduces all the 3-crossings to 2-crossings. We obtain the same diagram if, after the move, the over-arc is again picked up and laid down inside $T_2'$.

\begin{figure}[htbp]
\vspace*{13pt}
\centerline{\includegraphics*[scale=.4]{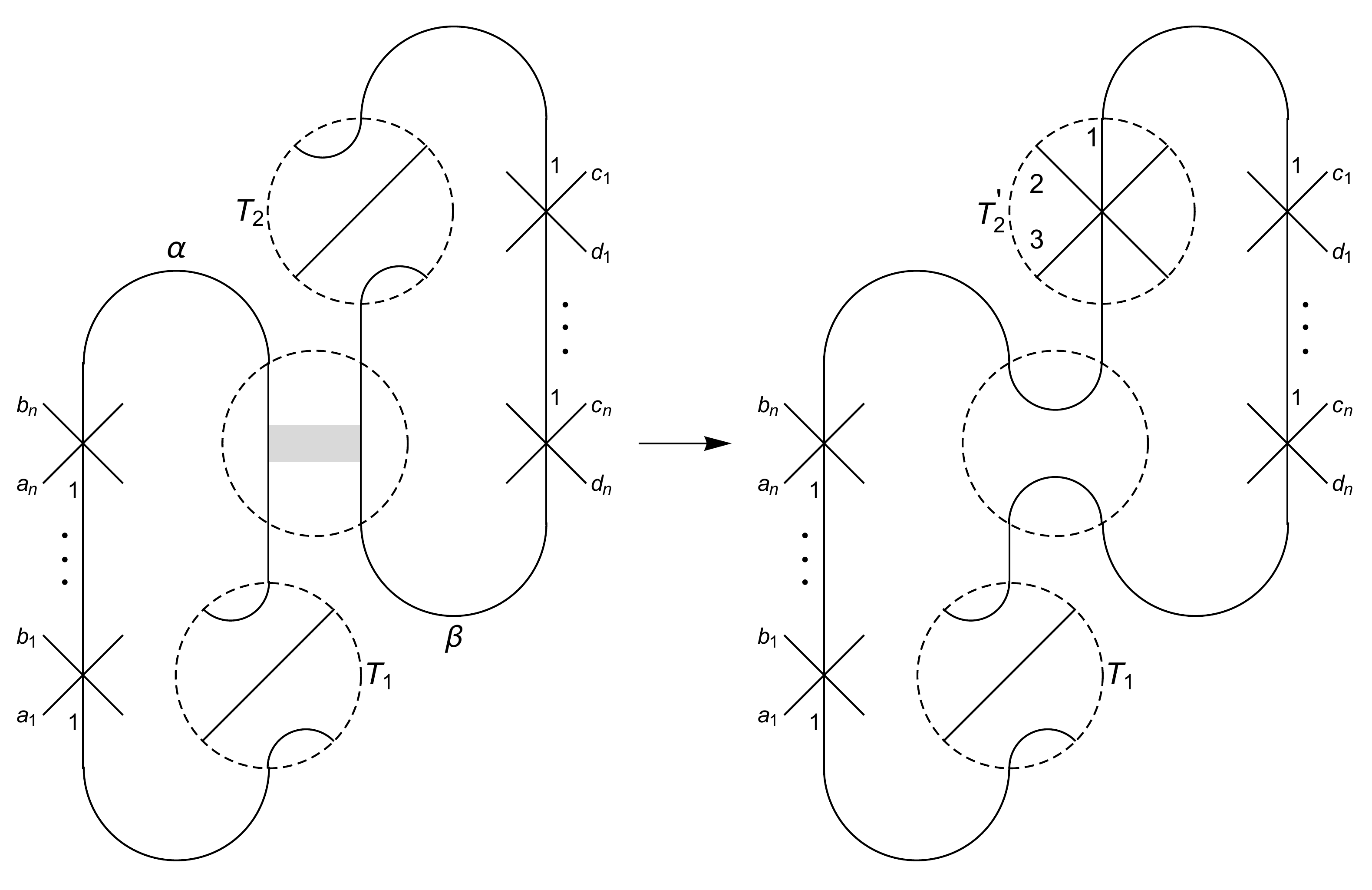}}
\caption{A possible band move.}
\label{bandMove}
\end{figure}

Figure~\ref{bandMove} depicts a possible band move. Suppose that $D$ contains two trivial 3-string tangles $T_1$ and $T_2$, each consisting of three parallel arcs, two disjoint over-arcs $\alpha$ and $\beta$, and a potential band move between $\alpha$ and $\beta$. (Recall that if $\gamma$ and $\delta$ are a pair of disjoint embedded arcs in the plane then a band move on $\gamma$ and $\delta$ is defined as follows. Consider an embedding of $I\times I$  (the band) into the plane such that $(\gamma \cup \delta)\cap (I \times I)=\partial I \times I$.  We can now alter $\gamma \cup \delta$ by replacing $\partial I \times I$ with $I \times \partial I$.)  Futhermore, suppose the endpoints of $\alpha$ are a pair of opposite endpoints in $T_1$, not the ends of the central arc, and the same is true of $\beta$ with respect to $T_2$. To perform the move, we perform the band move, joining $\alpha$ to $\beta$. Additionally, we replace $T_2$ with a 3-crossing $T_2'$ where the central arc of $T_2$ becomes the lowest strand of $T_2'$, the endpoints of $\beta$ are joined to form the highest strand of $T_2'$, and the other opposite pair of endpoints are joined to form the middle strand of $T_2'$. As with the basepoint move, the two 3-diagrams can be seen to represent the same link if we pass to 2-diagrams by lifting up the over-arcs and laying them down as the over-crossing strand of a single 2-crossing. Again, the situation need not look exactly as in Figure~\ref{bandMove}.

In addition to the moves illustrated in Figures~\ref{1and2-move},  \ref{basepointMove} and \ref{bandMove}, we include one more move. Let $\alpha$ be an arc of  a 3-diagram $D$ that passes through  no crossings and suppose $\alpha'$ is another arc whose interior does not intersect $D$, which has the same boundary points as $\alpha$, and which has no self-crossings. Replacing $\alpha$ with $\alpha'$ is called a  {\it trivial pass move}. If $D$ is connected, and we think of the diagram as lying on a 2-sphere rather than a plane, then the trivial pass move can be achieved by isotopy of the diagram. But if $D$ is disconnected, then  trivial pass moves may be nontrivial. For example, trivial pass moves allow one subdiagram of $D$ to be picked up and put down in a different complementary region of the rest of $D$. With classical diagrams, this can be accomplished with Reidemeister moves, but for 3-diagrams, a trivial pass move may not be a consequence of our other moves.

We may now state our main theorem:

\begin{theorem}\label{3-diagram moves}
Two  unoriented 3-diagrams $D_1$ and $D_2$ are related by a sequence of 1-moves, 2-moves, basepoint moves,  band moves, and trivial pass moves, if and only if  natural orientations on $D_1$ and $D_2$  define the same oriented link, up to reversal of maximal nonsplit sublinks.
\end{theorem}

If $D_1$ and $D_2$ represent nonsplit links, the trivial pass move is not needed and the statement of the theorem can be made somewhat simpler.

\begin{corollary}\label{3-diagram moves on nonsplit links} Two unoriented 3-diagrams $D_1$ and $D_2$ of nonsplit links are related by a sequence of 1-moves, 2-moves, basepoint moves,  and band moves, if and only if  natural orientations on $D_1$ and $D_2$ define the same oriented link, up to complete reversal.
\end{corollary}
 
 \begin{corollary}\label{3-diagram moves on knots}  Two unoriented 3-diagrams of knots are related by a sequence of 1-moves, 2-moves, basepoint moves,  and band moves, if and only if they define  the same unoriented knot.
\end{corollary}

{\bf Remark:} Notice that a number of variations on all of the 3-diagram moves exist that still preserve the underlying link with natural   orientation. For example, in the 1-move, the 1-gon could lie under the other strand. In the 2-move, the central strand could lie between or below the other two strands.  Notice that in both the basepoint move and the band move, the over-arcs could be replaced with {\it under-arcs} instead (that is, arcs that pass beneath all other crossings). In Figures~\ref{basepointMove} and \ref{bandMove}, all the heights of 1 would then become 3, etc.. These variations on the moves still preserve the link type,  but are not needed in the statement of Theorem~\ref{3-diagram moves}. Instead, it is possible to derive these variations from the other moves.

We will prove Theorem~\ref{3-diagram moves} in Section~\ref{proof of main theorem}.

\section{State Markers and the Even State}

 Given a classical link projection, each crossing divides the plane locally into four regions. A {\it state marker} at a given crossing is a choice of two opposite regions indicated by placing dots in the corners of the  two regions near the crossing. A {\it state} for the projection is a choice of state marker at each crossing. An {\it even} state is one where each complementary region of the projection contains an even number of dots.
  
A state marker at a crossing determines a way to smooth that crossing, namely, smooth it in the way that connects the two regions marked with the dots. If a crossing is a {\it self-crossing}, that is, the two strands there belong to the same component, then exactly one of the two state markers will determine a smoothing which splits the component into two components. The other state marker determines a smoothing which will maintain a single component. Call the former the  {\it fission} state marker.
 
If the projection $D$ is oriented, then we may use the orientation to determine a state marker at each crossing as illustrated in Figure~\ref{orientation state}. If the crossing is a self-crossing, then the orientation induced state marker and the fission state marker coincide. If we choose the state marker  at every crossing by using the orientation, we call this the {\it orientation induced state} of $D$. 
\begin{figure}[htbp]
\vspace*{13pt}
\centerline{\includegraphics*[scale=.15]{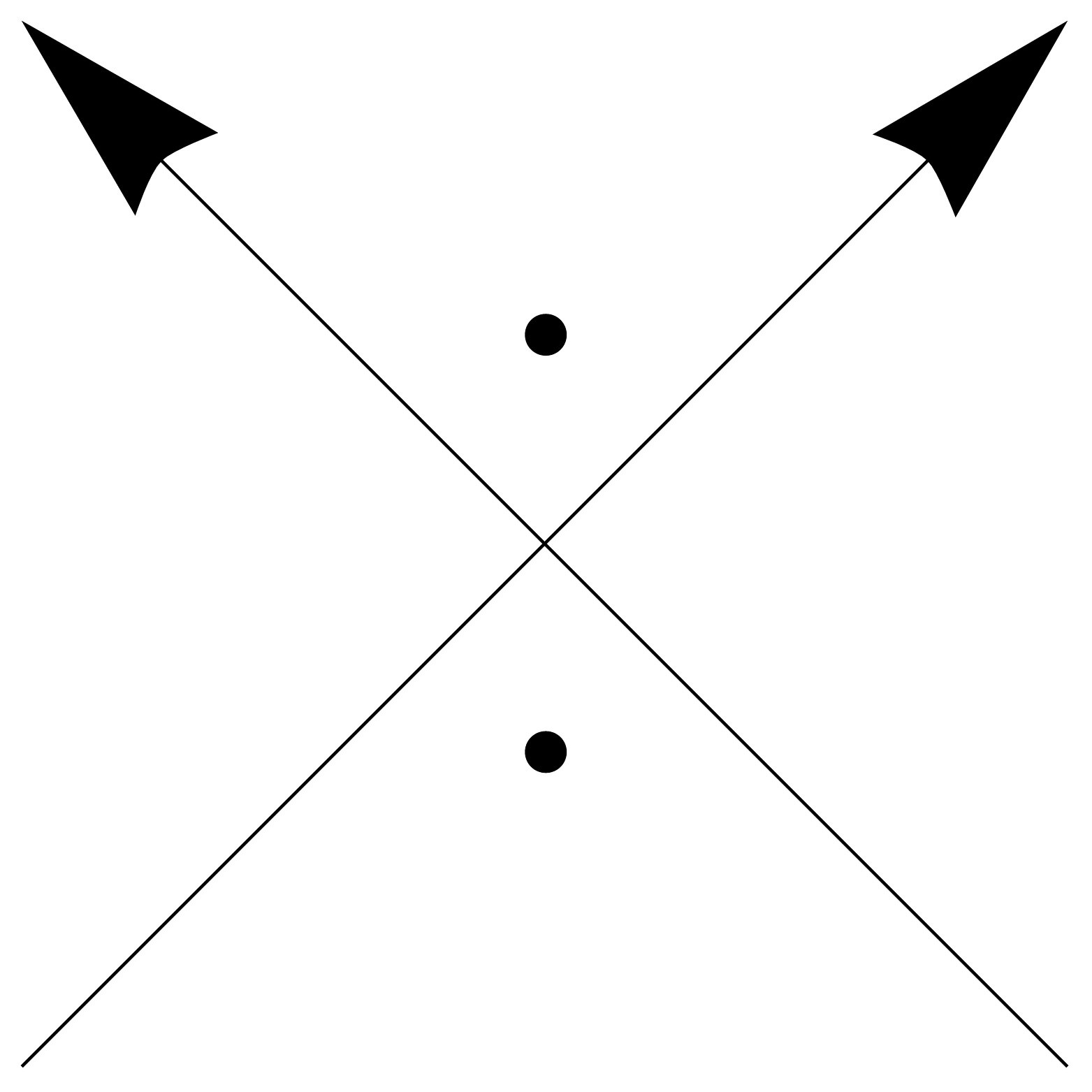}}
\caption{The state marker induced by the orientation.}
\label{orientation state}
\end{figure}

\pagebreak
\begin{theorem}\label{even states}
Suppose $D$ is a classical link projection. Then all the following are true.
\begin{enumerate}
 \item\label{orientation induced states are even} An orientation induced state of $D$ is even.
\item\label{orientations that induce the same state} If $D$ is connected, then orientations ${\mathcal O}_1$ and ${\mathcal O}_2$ of $D$ induce the same state if and only if they are equal or one is the reverse of the other.
\item\label{even states  are orientation induced} Every even state of $D$ is the orientation induced state for some orientation of $D$.
\item\label{number of states} If $D$  is a connected diagram of a link with $k$  components, then $D$ has $2^{k-1}$ distinct even states.
\end{enumerate}
\end{theorem}
\begin{proof}
To prove (\ref{orientation induced states are even}), suppose $D$ is oriented, that $S$ is the orientation induced state, and that $R$ is a complementary region of $D$.  As we traverse the boundary of $R$, we encounter an even number of vertices where the orientation of the edges on $\partial R$ reverses from clockwise to counter-clockwise or vice versa. These are exactly the vertices that contribute a dot to $R$. Thus $R$ contains an even number of dots and $S$ is an even state. 

For~(\ref{orientations that induce the same state}), because reversing all orientations in Figure~\ref{orientation state} does not change the placement of the dots, it follows that if one orientation is the reverse of the other then both induce the same state. Conversely, if orientations  ${\mathcal O}_1$ and ${\mathcal O}_2$ induce the same state, then at each crossing they either agree or are the reverse of each other. Because $D$ is connected, it follows that they are equal at every crossing, or opposite at every crossing. (This is false if $D$ is not connected.)

Let $S$ be an even state of $D$. We will prove (\ref{even states  are orientation induced}) by induction on the number of double points in $D$. If $D$ has one double point, the result is obvious. Now suppose that $D$ has more than one double point. Pick one and smooth it as in Figure~\ref{smoothing}, to obtain a new projection $D'$.
\begin{figure}[htbp]
\vspace*{13pt}
\centerline{\includegraphics*[scale=.4]{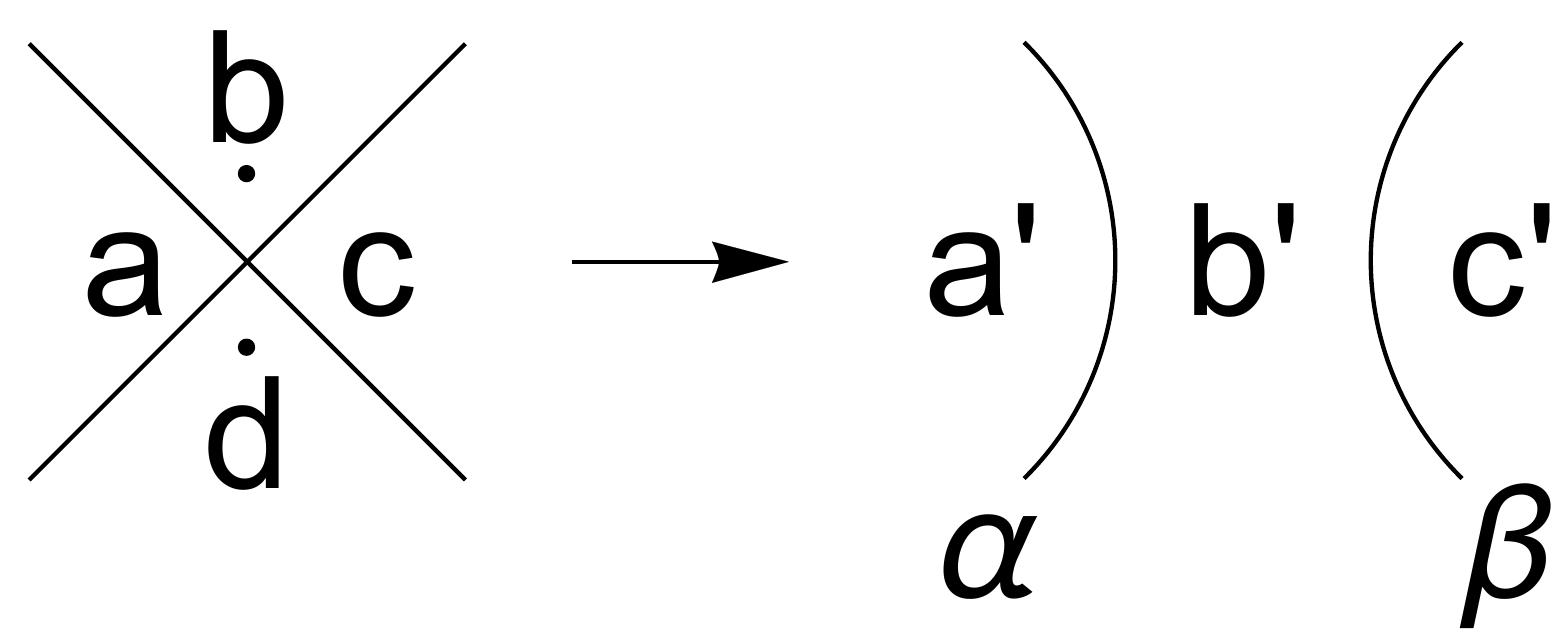}}
\caption{Smoothing a crossing.}
\label{smoothing}
\end{figure}
By forgetting the state marker at the crossing that was removed by smoothing, we obtain a state $S'$ of $D'$. The state $S'$ is even because the regions $b$ and $d$ each contain an even number of dots and, hence, so does the region $b'$ in $D'$. Because $D'$ has one fewer crossing than $D$, there exists an orientation of $D'$ for which $S'$ is the induced state. If $b$ and $d$ are really the same region, then $D'$ is disconnected with $\alpha$ and $\beta$ lying in disjoint subdiagrams of $D'$. If necessary, we may reverse all the orientations of the subdiagram containing $\alpha$ so that $\alpha$ and $\beta$ are oriented parallel to each other. As already proven,  this will not change the even state $S'$. If instead,   $b$ and $d$ are distinct regions, then  we may traverse the boundary of $b'$, starting at $\alpha$ and reach $\beta$. One way takes us essentially around the boundary of $b$, the other way around the boundary of $d$. But in either case, we pass an odd number of dots because both $b$ and $d$ each contain an even number of dots. As we pass each dot, the orientation of the edges on the boundary of $b'$ is reversed. Thus when we get to $\beta$ we see that it is oriented parallel to $\alpha$. Hence, in both cases, there is an orientation of $D'$ for which $S'$ is the induced state and moreover, $\alpha$ and $\beta$ are parallel. We may now use this orientation to obtain an orientation of $D$ that induces $S$.

Finally, for~(\ref{number of states}), if $D$ is a connected projection of a link with $k$ components, we now see that the map from the set of orientations of $D$ to the set of even states of $D$ is a 2--\,to\,--1 surjection. The total number of orientations of $D$ is $2^k$ and hence the total number of even states is $2^{k-1}$. 
\end{proof}

 In the case of a knot we have
   \begin{corollary}
If $D$ is a knot projection, then $D$ has a unique even state.
\end{corollary}

\section{Crossing Circles}\label{crossing circles}

Given a classical diagram $D$,  a {\it crossing circle} for $D$ is a circle $C$ embedded in the projection 2-sphere that  intersects $D$ only at crossings. Additionally, at each such crossing, the two strands of $D$ are each transverse to $C$.  The crossing circle is {\it trivial} if it contains no crossings of $D$. A finite set of disjoint crossing circles that together contain all crossings of $D$, and moreover, where a crossing  has been chosen on each nontrivial crossing circle,  will be called a {\it   crossing circle cover} of $D$, or more concisely, a $CCC$ of $D$. We will refer to the chosen crossing on each crossing circle as the {\it basepoint}. If no basepoints have been chosen on the crossing circles, we will refer to the collection of circles as an {\it unmarked} $CCC$.

Each $CCC$ of a diagram $D$ determines an even state of the underlying projection by choosing the state marker at each crossing that labels the  regions containing the crossing circle. Moreover, it is easy to see that every even state determines an umarked $CCC$, although not uniquely. Because there are an even number of dots in each region, we may connect them in pairs with disjoint arcs inside each region and then piece the arcs together to obtain an unmarked $CCC$. There may be more than one way to do this. However, all unmarked $CCC$'s associated to the same even state are related by  band moves  and the introduction or deletion of a trivial crossing circle.  Thus, Theorem~\ref{even states} implies that every diagram has  a $CCC$. In the case of a knot diagram, the even state is unique. Thus any two $CCC$'s of the same knot diagram differ by band moves, change of basepoints, and insertion or deletion of trivial crossing circles. For links, we must additionally require that the two $CCC$'s determine the same even state in order that they be so related. We summarize these statements in the following result.

\begin{lemma}\label{CCC's with the same even state}
Two $CCC$'s of a classical diagram $D$  determine the same even state of $D$ if and only if they differ by band moves, change of basepoints, and insertion or deletion of trivial crossing circles.
\end{lemma}

If we consider oriented  diagrams, then we can demand that any $CCC$ of the diagram be {\it compatible} with the orientation, that is, the  state determined by the $CCC$ matches the orientation induced  state. Lemma~\ref{CCC's with the same even state} now implies that if two $CCC$'s are both compatible with the same oriented diagram, then they differ by band moves, change of basepoints, and insertion or deletion of trivial crossing circles.

\begin{theorem}\label{CCC moves} Suppose that $D_1$ and $D_2$ are two  oriented 2-diagrams, each equipped with a compatible $CCC$. Then $D_1$ and $D_2$ represent the same oriented link, up to orientation reversal on maximal nonsplit sublinks, if and only if one can be changed to the other by the following moves:
\begin{enumerate}
\item Reversing the orientation on any subdiagram of a diagram.
\item Insertion or deletion of a trivial crossing circle.
\item The four Reidemeister moves, and their inverses,  on compatible $CCC$-equipped diagrams  shown in Figure~\ref{CCC R moves}.
\item Changing the $CCC$ by moving the basepoint on one crossing circle.
\item Changing the $CCC$ by a band move. If the band move splits a nontrivial crossing circle into two nontrivial crossing circles, we must introduce a new basepoint. If the band move joins two nontrivial crossing circles together, then one basepoint is removed.
\end{enumerate}
\end{theorem}
\begin{figure}[htbp]
\vspace*{13pt}
\begin{center}

\begin{tabular}{c}
\begin{tabular}{ccc}
\includegraphics*[scale=.3]{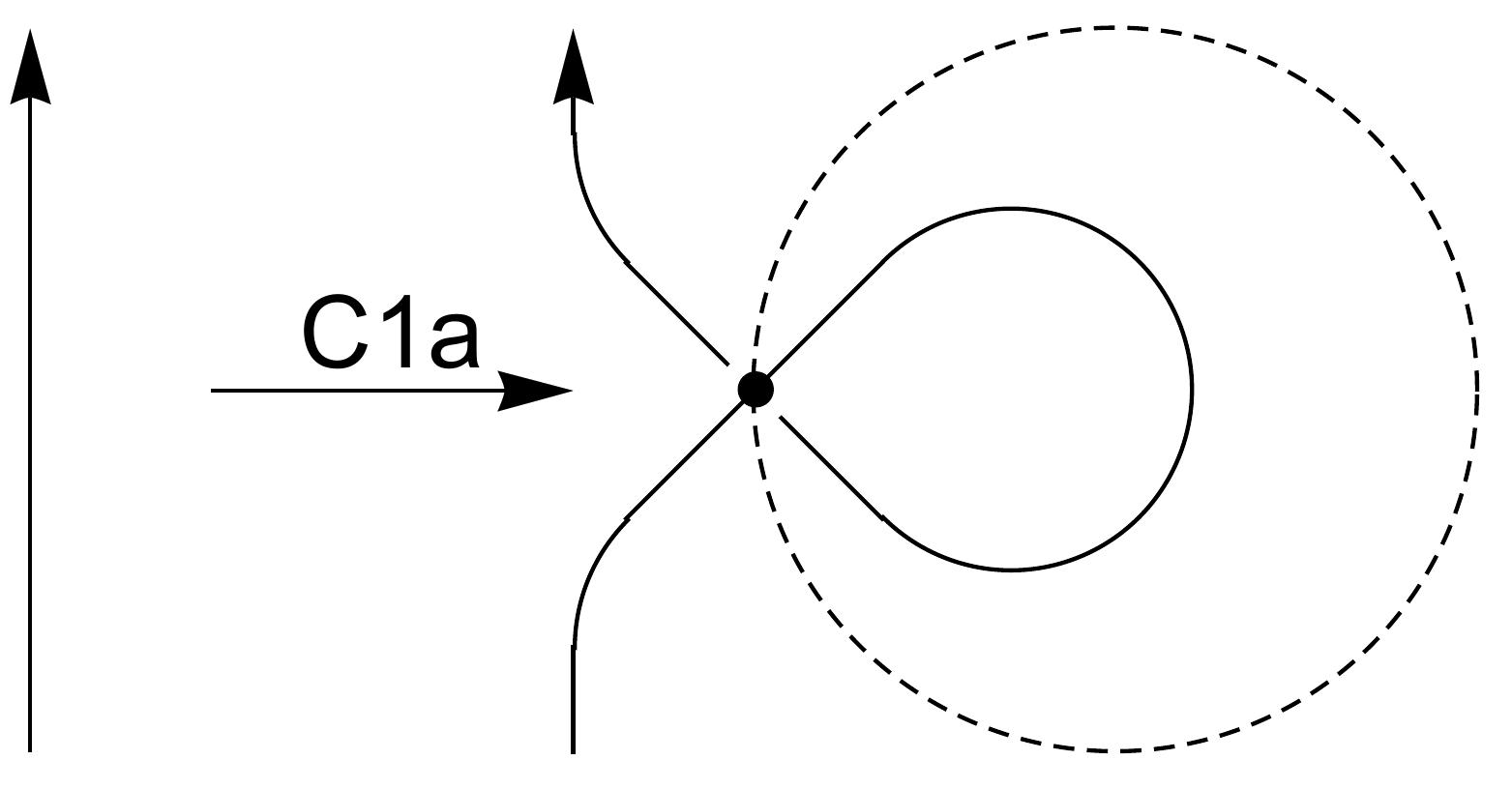}&\quad\quad\quad&
\includegraphics*[scale=.3]{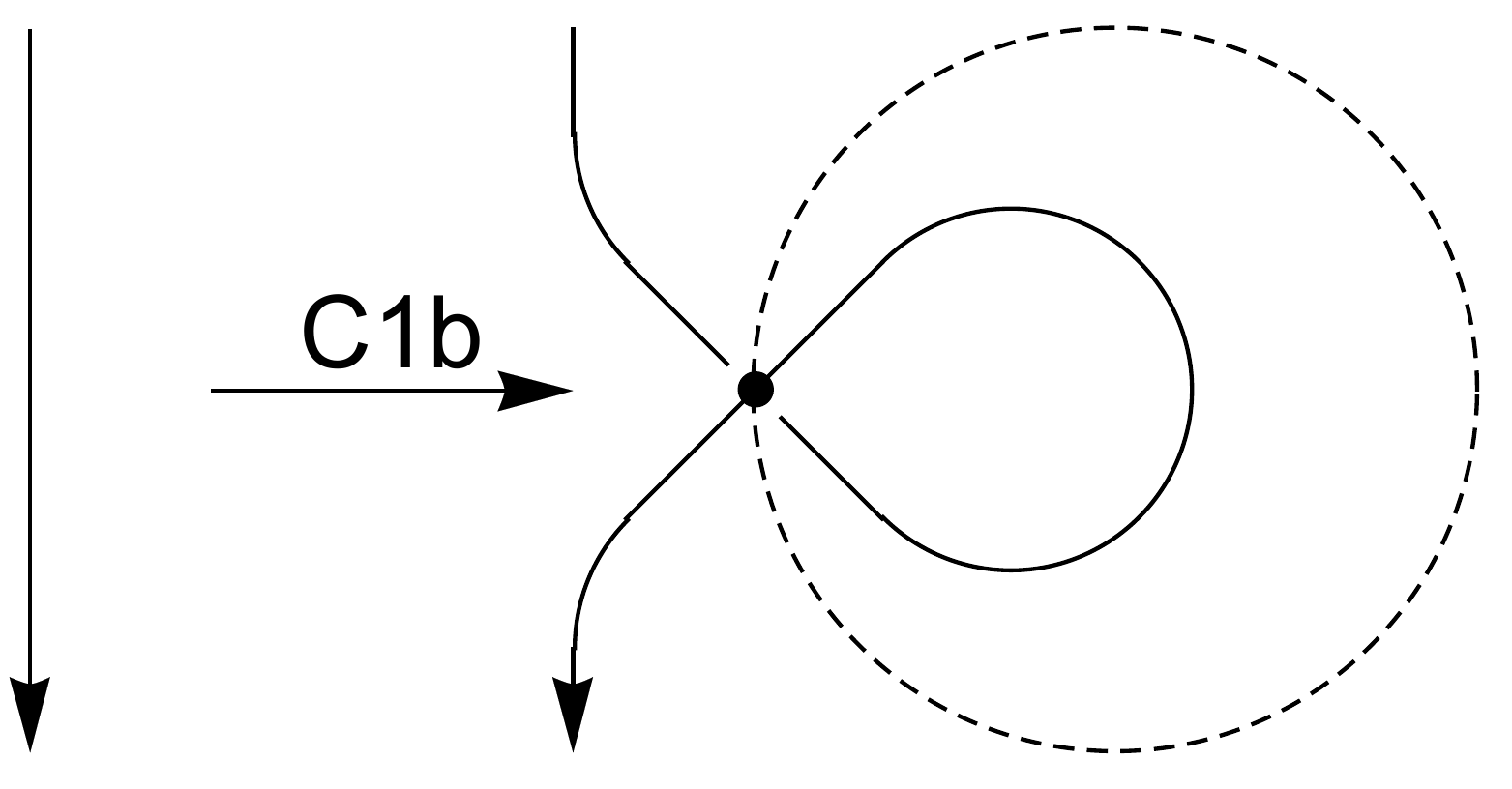}
\end{tabular}\\
\begin{tabular}{ccc}
\includegraphics*[scale=.3]{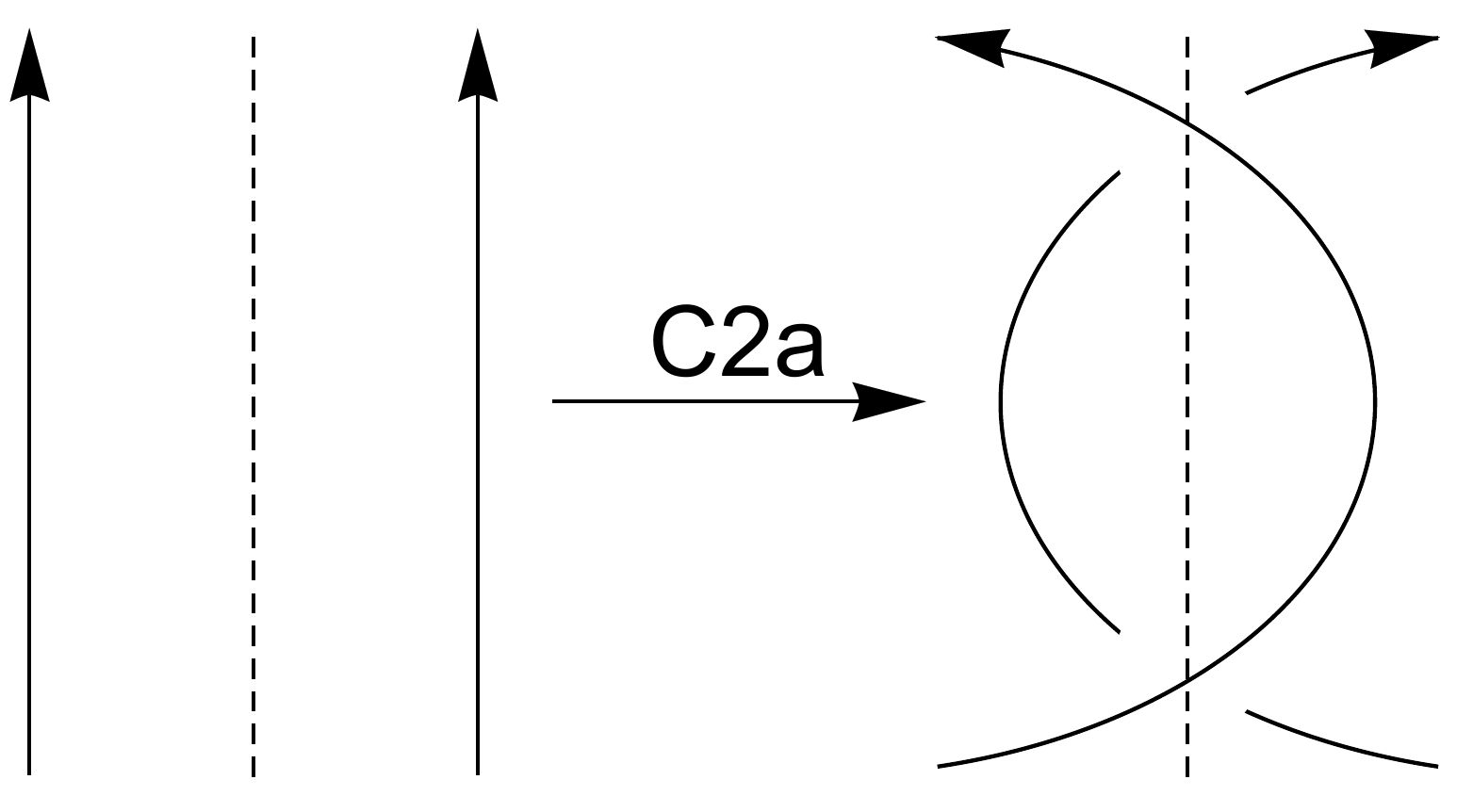}&\quad\quad\quad&
\includegraphics*[scale=.4]{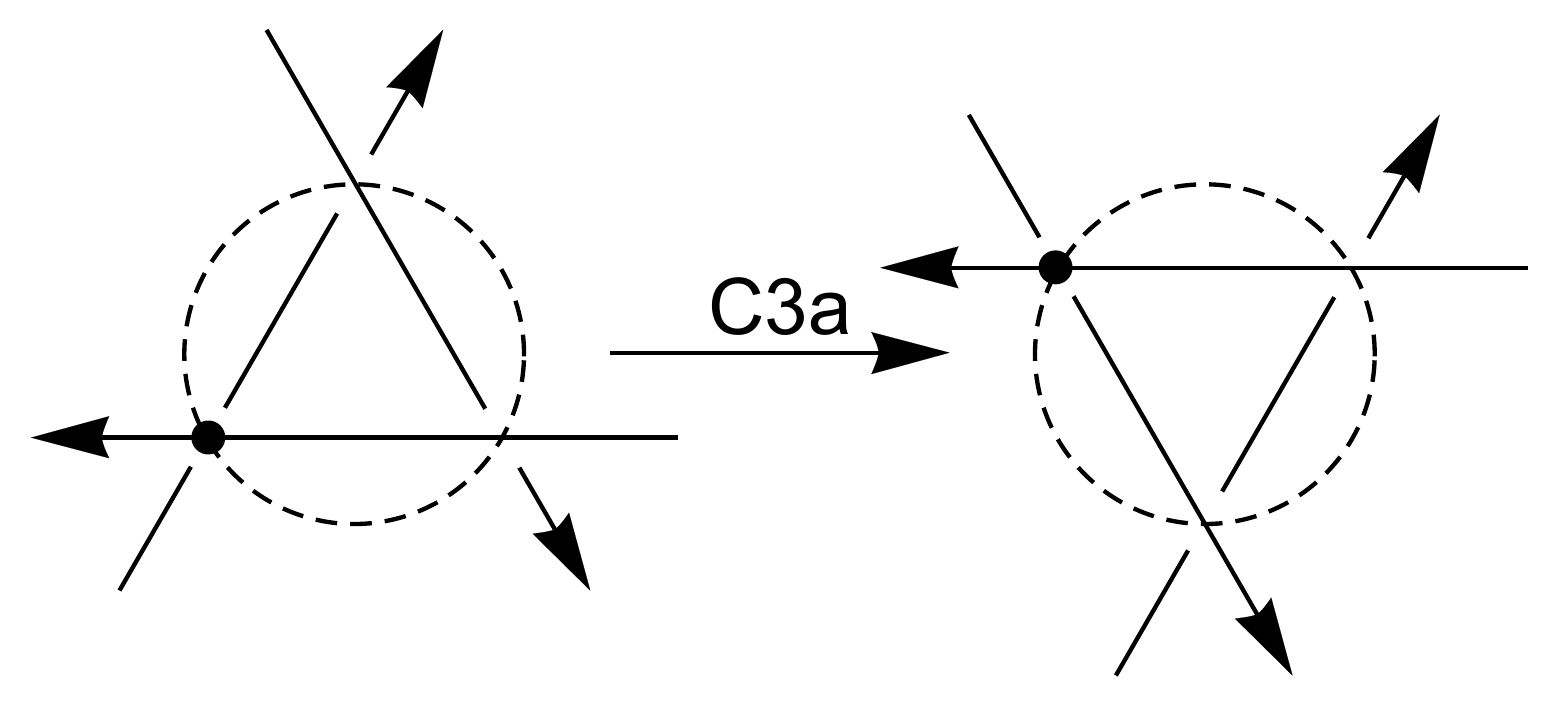}
\end{tabular}\end{tabular}
\end{center}
\caption{Reidemeister moves on oriented diagrams with compatible $CCC$'s. Move $C2a$ is shown in the case where the crossing circle on the left is nontrivial. The case with trivial crossing circle, if shown, would include a basepoint on one of the two crossings introduced by the move.    \hfill}
\label{CCC R moves}
\end{figure}

\begin{proof} 
We first show that if two diagrams are related by any one of the five kinds of moves, then the oriented links that they define  are the same up to orientation reversal of maximal nonsplit sublinks. In the case of the first move, notice that reversing the orientation on one subdiagram of a diagram reverses the orientation of a union of maximal nonsplit sublinks. Notice that the even state induced by the orientation is unchanged, so the $CCC$ remains compatible. Each of the other moves preserves the oriented link exactly.

Now suppose that $D_1$ and $D_2$ are oriented diagrams that define oriented links $L_1$ and $L_2$, respectively, that differ only by reversing maximal nonsplit sublinks. In general, if $L$ is an oriented link, let $|L|$ be the unoriented link obtained from $L$ by forgetting the orientation. Thus $|L_1|=|L_2|$. Let $D$ be some unoriented diagram of $|L_1|$ such that every subdiagram of $D$ defines a maximal non-split sublink of $|L_1|$. We may now orient $D$ somehow to obtain an oriented diagram $\overline D_1$ representing $L_1$ and perhaps some other way to obtain an oriented diagram $\overline D_2$ representing $L_2$. Notice that $\overline D_1$ and $\overline D_2$ differ by at most complete reversal on subdiagrams. Because $D_1$ and $\overline D_1$ are two oriented diagrams of the same oriented link $L_1$, there exists a sequence of $\Omega 1a, \Omega 1b, \Omega 2a$, and $\Omega 3a$ moves and their inverses that take $D_1$ to $\overline D_1$. The same is true for $D_2$ and $\overline D_2$.  Our goal, for both $i=1$ and $i=2$, is to replace these  oriented Reidemeister moves with the corresponding moves of Figure~\ref{CCC R moves}, thus carrying the $CCC$ from $D_i$ to one which is compatible with $\overline D_i$. However, we will see that to do this may also require moves (1), (2), (4), or (5) of the theorem. Once we show how to do this, we can then apply move (1) to $\overline D_1$ to obtain a $CCC$ on $\overline D_2$. We now have two $CCC$'s on $\overline D_2$: the one coming from $D_1$ via $\overline D_1$ and the one coming from $D_2$. But both are compatible with the same oriented diagram and hence, by the remark following Lemma~\ref{CCC's with the same even state},  must be related by moves (2), (4), or (5). Thus, we will finally obtain a path from $D_1$ to $D_2$ using the moves given in the theorem.

We now return to the problem of ``extending'' the sequence of oriented Reidemeister moves from $D_1$ to $\overline D_1$.
It suffices to consider the case where $d$ is an oriented diagram with compatible $CCC$ and we wish to perform a single Reidemeister move (or its inverse)  from the set  $\{\Omega 1a, \Omega 1b, \Omega 2a, \Omega 3a\}$ to $d$, changing it into $d'$. If the $CCC$ for $d$ does not appear exactly as depicted in Figure~\ref{CCC R moves}, then we may need to first change the $CCC$ by moves (2), (4), or (5).  For example, suppose a crossing $c$ is to be eliminated by the inverse of $\Omega 1a$ or $\Omega 1b$ but the crossing circles do not appear exactly as in $C 1a$ or $C 1b$, respectively. If trivial crossing circles lie within the 1-gon, then we may remove them with move (2). Or,  if the crossing circle that contains $c$ also contains other crossings, then a band move can be used to split off a crossing circle that contains only $c$. On the other hand, if we wish to perform $\Omega 1a$ or $\Omega 1b$, we can immediately apply $C 1a$ or $C1b$; the $CCC$ for $d$ already appears as is needed.

The situation is a bit more subtle when extending an $\Omega 2a$ move to a $C 2a$ move. Let  $\alpha$ and $\beta$ be  the two strands involved in the Reidemeister move. There may be strands of crossing circles that lie ``in between'' $\alpha$ and $\beta$. Using band moves, we can clear out this region, one pair of crossing circle strands at a time, leaving either one or no crossing circles between them. If one crossing circle strand remains then we may perform the $C2a$ move  (and introduce a basepoint at one of the two new crossings if necessary). If no crossing circles lie in between $\alpha$ and $\beta$ then it must be because 
$\alpha$ and $\beta$  lie in different subdiagrams of $D$. For if they are both part of the same subdiagram, let $R$  be the complementary region that lies between   $\alpha$ and $\beta$. Note that $R$ may contain  subdiagrams  in its interior. Let $S$ be the set of all crossings $c$ of the boundary of $R$ such that the crossing circle that contains $c$ locally meets the interior of $R$. As we traverse any  boundary component of $R$, either the one containing both $\alpha$ and $\beta$, or perhaps one belonging to another subdiagram, the orientations on the edges of $\partial R$ reverse each time we pass a crossing in $S$. Hence, each boundary component of $R$ contains an even number of crossings in $S$. Moreover, as we traverse the boundary component of $R$ traveling from $\alpha$ only as far as $\beta$, we must pass an odd number of crossings in $S$. Therefore, any arc from $\alpha$ to $\beta$ lying inside $R$ and intersecting crossing circles transversely must  intersect an odd number of crossing circles. Hence, if no crossing circles lie between $\alpha$ and $\beta$, then they lie in different subdiagrams and we may introduce a trivial crossing circle that separate $\alpha$ and $\beta$.   The $CCC$ now appears as needed to perform the  $C2a$ move. In this case, a basepoint must be added to one of the newly introduced crossings. Finally, if we wish to extend the inverse of $\Omega 2a$ to the inverse of $C2a$, and the crossing circle that contains the two crossings  also contains other crossings, we can first move the basepoint, if necessary, to not lie at either of the two crossings that are to be eliminated. If this is not the case, then the basepoint lies at one of the two crossings and after eliminating the 2-gon, the crossing circle is trivial.

Finally, notice that to extend $\Omega 3a$ (or its inverse)  to $C 3a$ (or its inverse) we can use band moves and deletion of trivial crossing circles, if necessary, to first alter the $CCC$ so that it appears as  shown in Figure~\ref{CCC R moves}.
\end{proof}

Using the fact that every classical diagram has a $CCC$, it is shown in  \cite{Adams5} that every link has a $3$-diagram. The construction is as follows. Given a  link $L$, let $D$ be a $2$-diagram of $L$ equipped with a  $CCC$. For each non-trivial crossing circle $C$, let $c$ be the chosen crossing on $C$. Now replace the over-crossing strand at $c$ with the crossing circle. That is, as we come into the crossing $c$ on the over-crossing strand, instead of continuing on the over-crossing strand, detour  around the crossing circle on top of all the other crossings on $C$ and in such a way as to eliminate the crossing $c$.  All of the other crossings on the crossing circle are now turned into 3-crossings.  If we do this for each crossing circle in the  $CCC$, we obtain a 3-diagram of the same link. Note that trivial crossing circles are ignored or, if one prefers, removed at the beginning of the process. An example is shown in Figure~\ref{CCC construction}. Alternatively, we could replace the under-crossing strand at $c$ by detouring around $C$ {\sl beneath} all the other crossings along $C$. We call the former the {\it over CCC} construction, and the latter the {\it under CCC} construction. If the $CCC$ contains  more than one circle, we could even mix the two constructions. Notice that in any case,  one 2-crossing has been lost for each crossing circle. This implies that  the 3-crossing number of a link is strictly less than the 2-crossing number. Let $\Psi$ be the map from the set of 2-diagrams equipped with  $CCC$'s into the set of 3-diagrams which is defined by the over $CCC$ construction.

If $D$ is an oriented 2-diagram with compatible $CCC$, then the over $CCC$ construction produces an oriented 3-diagram.  By abuse of notation, we will still use $\Psi$ to refer to this construction in the oriented case. 

The following lemma seems worth recording, but is not needed in our exposition.
\begin{lemma} If $D$ is a connected   oriented 2-diagram with compatible $CCC$, then the orientation on $\Psi(D)$ is  natural. 
\end{lemma}
\begin{proof} Suppose  that $D$ is a connected oriented 2-diagram with compatible $CCC$  and suppose $C$ is a crossing circle for $D$. Consider two  consecutive crossings $c_1$ and $c_2$ as one traverses $C$. Thinking of these as 3-crossings (two strands from $D$ together with $C$), then
if we orient $C$ so that $c_1$  is Type A, then $c_2$  must also be of Type A. To see this,   let $R$ be the complementary region  of $D$ containing the arc of $C$ connecting $c_1$ and $c_2$. As we traverse the boundary of $R$ from $c_1$ to $c_2$, which is possible because $D$ is connected,  the orientation of $\partial R$ reverses each time we pass the endpoint of a crossing circle inside $R$. Thus there are an even number of reversals between $c_1$ and $c_2$ and hence $c_2$ is also of Type A. Thus if one crossing on $C$ is Type A, all the crossings on $C$ are Type A. Now when we perform the over $CCC$ construction, the orientation induced on $C$ by detouring around $C$ instead of following the over-crossing strand at the basepoint, is the one that makes all the crossings along $C$ of Type A.
\end{proof}

\begin{figure}[htbp]
\vspace*{13pt}
\centerline{\includegraphics*[scale=.3]{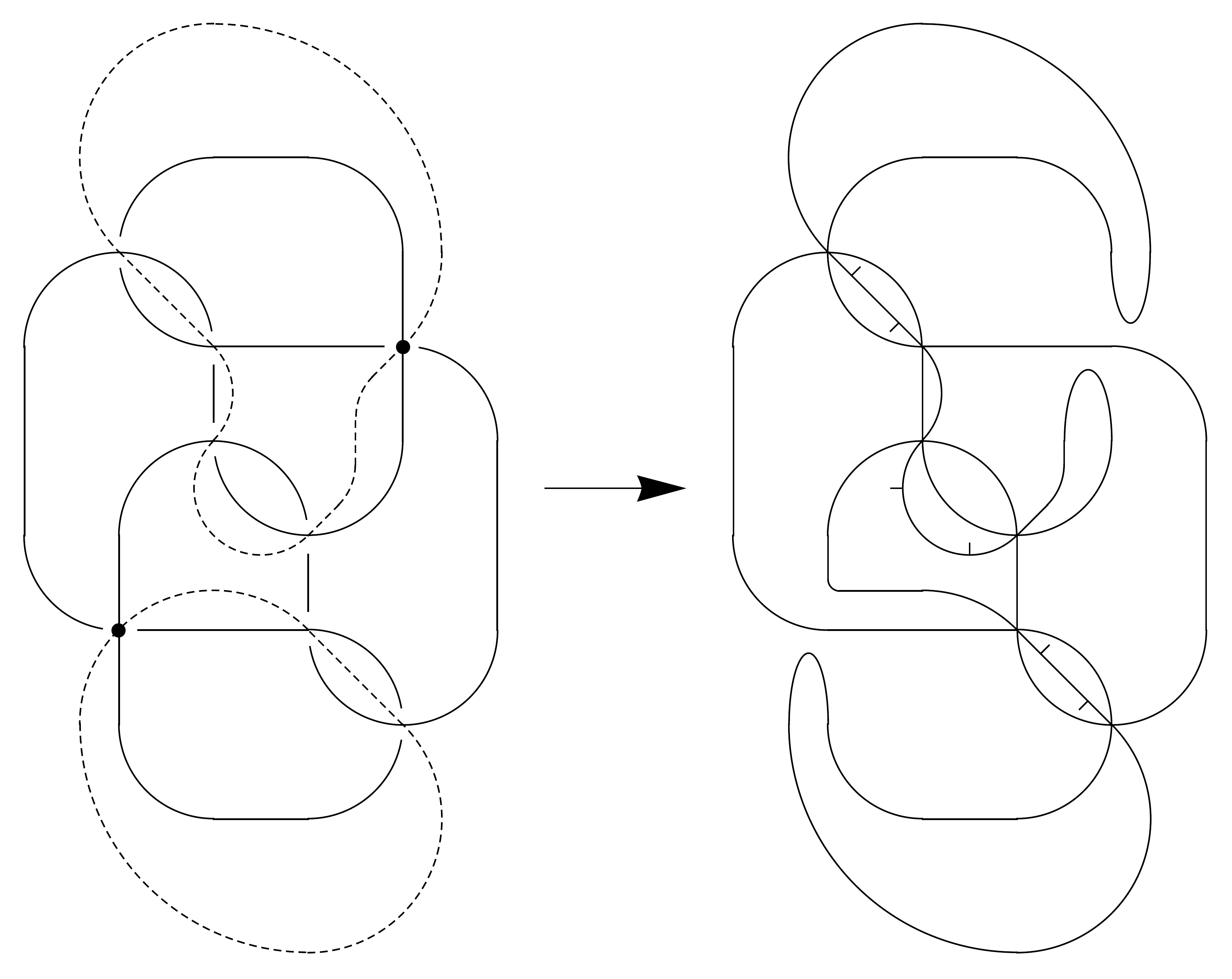}}
\caption{Using a CCC to turn a 2-diagram of a knot  into a 3-diagram. Marks at each 3-crossing indicate the highest strand and point to the middle-level strand.}
\label{CCC construction}
\end{figure}

\section{Proof of Theorem~\ref{3-diagram moves}}\label{proof of main theorem}

\begin{lemma}\label{1-move to image of Psi} Let $D$ be any 3-diagram with a piecewise natural orientation. Then there exists an oriented 2-diagram $d$ with compatible $CCC$ such that $D$ and $\Psi(d)$ differ by 1-moves. 
\end{lemma}
\begin{proof}  Let $D$ be any 3-diagram of a link with a piecewise natural orientation. We may resolve each 3-crossing of $D$ into three 2-crossings to obtain an oriented 2-diagram $d$ of the same oriented link.  Because of Lemma~\ref{natural iff Type A},  a compatible $CCC$ for $d$ exists with one crossing circle for each such set of three crossings, as shown  in Figure~\ref{1-move to get into image of Psi}. The diagram $\Psi(d)$ and $D$ are now related by a 1-move near each 3-crossing of $D$ as shown in  the figure. (A similar figure can be used in the case where the heights of the three strands in $D$ are different than pictured.)
\end{proof}

\begin{figure}[h]
\vspace*{13pt}
\centerline{\includegraphics*[scale=.4]{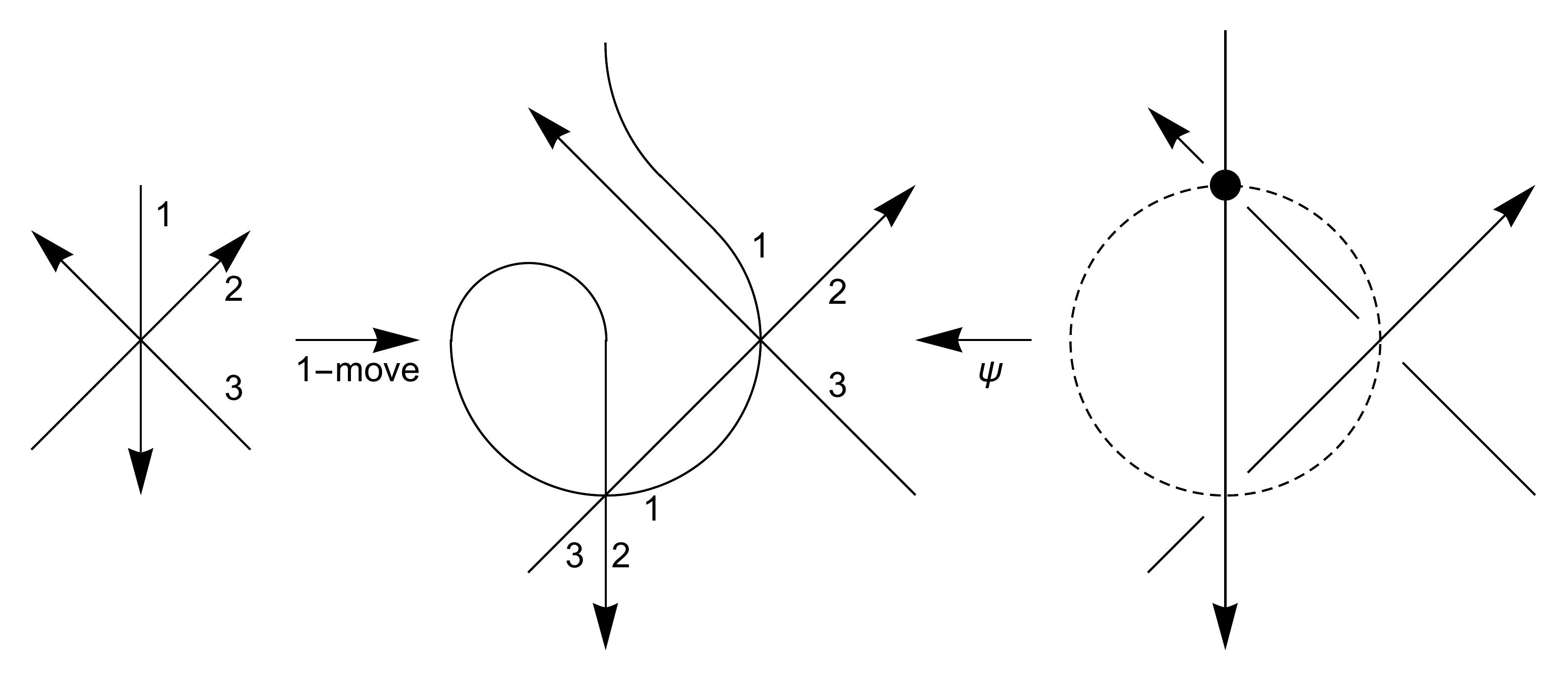}}
\caption{Changing the 3-diagram $D$ on the left, with piecewise natural orientation, into a 3-diagram in the image of $\Psi$.}
\label{1-move to get into image of Psi}
\end{figure}

\noindent{\it Proof of Theorem~\ref{3-diagram moves} and Corollaries~\ref{3-diagram moves on nonsplit links} and \ref{3-diagram moves on knots}:} Suppose first that $D_1$ and $D_2$ are two 3-diagrams related by a sequence of the moves given in Theorem~\ref{3-diagram moves}. It is not hard to see that the 1-move, 2-move, basepoint move and band move may be extended to 3-diagrams with checkerboard colorings. Thus 3-diagrams related by these moves have natural orientations that are either the same or differ by complete reversal. However, if $D_1$ is not connected and the trivial pass move is used to move a strand to the other side of a subdiagram, natural orientations of $D_1$ and $D_2$ may now differ by the reversal of a union of maximal nonsplit sublinks of $D_1$. This completes the proof of one direction of the theorem as well as one direction of both Corollaries~\ref{3-diagram moves on nonsplit links} and \ref{3-diagram moves on knots}.

Conversely, suppose that $D_1$ and $D_2$ are two 3-diagrams each having a  natural orientation defining oriented links $L_1$ and $L_2$, which are the same, up to reversal of  maximal nonsplit sublinks. By Lemma~\ref{1-move to image of Psi}, we may use 1-moves to change each 3-diagram $D_i$ into a 3-diagram $D_i'=\Psi(d_i)$ where $d_1$ and $d_2$ are  oriented 2-diagrams of $L_1$ and $L_2$, respectively, each with compatible $CCC$.   We may now change $d_1$ into $d_2$ by a sequence of the moves given in Theorem~\ref{CCC moves}. Suppose this gives the sequence $d_1=e_1, e_2, \dots, e_k=d_2$ where each $e_i$ is an oriented 2-diagram equipped with a compatible $CCC$. It remains only to show that if $e_i$ is taken to $e_{i+1}$ by one of the moves given in Theorem~\ref{CCC moves}, then $\Psi(e_i)$ can be taken to $\Psi(e_{i+1})$ by a sequence of the 3-diagram moves given in Theorem~\ref{3-diagram moves}. We will consider each type of possible move.

 If  $e_i$ and $e_{i+1}$ are related by either the insertion or deletion of a trivial crossing circle, or $C 1a$ or $C 1b$, then  $\Psi(e_i)=\Psi(e_{i+1})$. Now suppose $e_i$ and $e_{i+1}$ are related by $C 2a$ . If the crossing circle lying between the two parallel arcs on the left is trivial, then a basepoint must be located on one of the two crossings on the right. It now follows that $\Psi(e_i)$ and $\Psi(e_{i+1})$ are related by a 1-move and a trivial pass move. If instead, the crossing circle is not trivial, then $\Psi(e_i)$ and $\Psi(e_{i+1})$ are related by a 2-move. Now suppose $e_i$ and $e_{i+1}$ are related by $C 3a$, then $\Psi(e_i)$ and $\Psi(e_{i+1})$ are related by a sequence of two 1-moves.

Next, if $e_i$ and $e_{i+1}$ are related by a $CCC$ change of basepoint, then the reader can verify that   $\Psi(e_i)$ and $\Psi(e_{i+1})$ are related by a 3-diagram basepoint move. Finally, suppose $e_i$ and $e_{i+1}$ are related by a $CCC$  band move. Now three cases can occur.  Without loss of generality, we may assume that one crossing circle is being split into two crossing circles by the move. If a trivial crossing circle is split into two trivial crossing circles, then $\Psi(e_i)$ and $\Psi(e_{i+1})$ are equal. If a non-trivial crossing circle is split into a nontrivial crossing circle and a trivial crossing circle, then $\Psi(e_i)$ and $\Psi(e_{i+1})$ are either equal or related by a trivial pass move. Finally, if a nontrivial crossing circle is split into two nontrivial crossing circles, then $\Psi(e_i)$ and $\Psi(e_{i+1})$ are related by a band move. 

Notice that the only time the trivial pass move is needed in the proof  is when one of $e_i$ or $e_{i+1}$ is a disconnected diagram. Hence, if we assume from the beginning that $D_1$ and $D_2$ are two 3-diagrams each defining nonsplit links, then this cannot occur. Thus we obtain Corollaries~\ref{3-diagram moves on nonsplit links} and \ref{3-diagram moves on knots}. \hfill $\square$

\section{Triple and Quintuple Crossing Numbers}\label{c_3 and c_5}

In any link diagram with multicrossings, notice that an $n$-crossing can always be increased to an $(n+2)$-crossing. After passing though the crossing on any of the $n$ strands, one can double back and pass through the crossing two more times before continuing on as before. This process is illustrated in Figure~\ref{threetofive}, where a 3-crossing is increased to a 5-crossing. Because every link has  both a 2-diagram as well as a 3-diagram, it follows that every link has an $n$-diagram for every $n\ge 2$. This allows us to define the $n$-crossing number of a link, denoted $c_n(L)$, as the minimum number of $n$-crossings in any $n$-diagram of $L$.

\begin{figure}[htbp]
\vspace*{13pt}
\centerline{\includegraphics*[scale=.6]{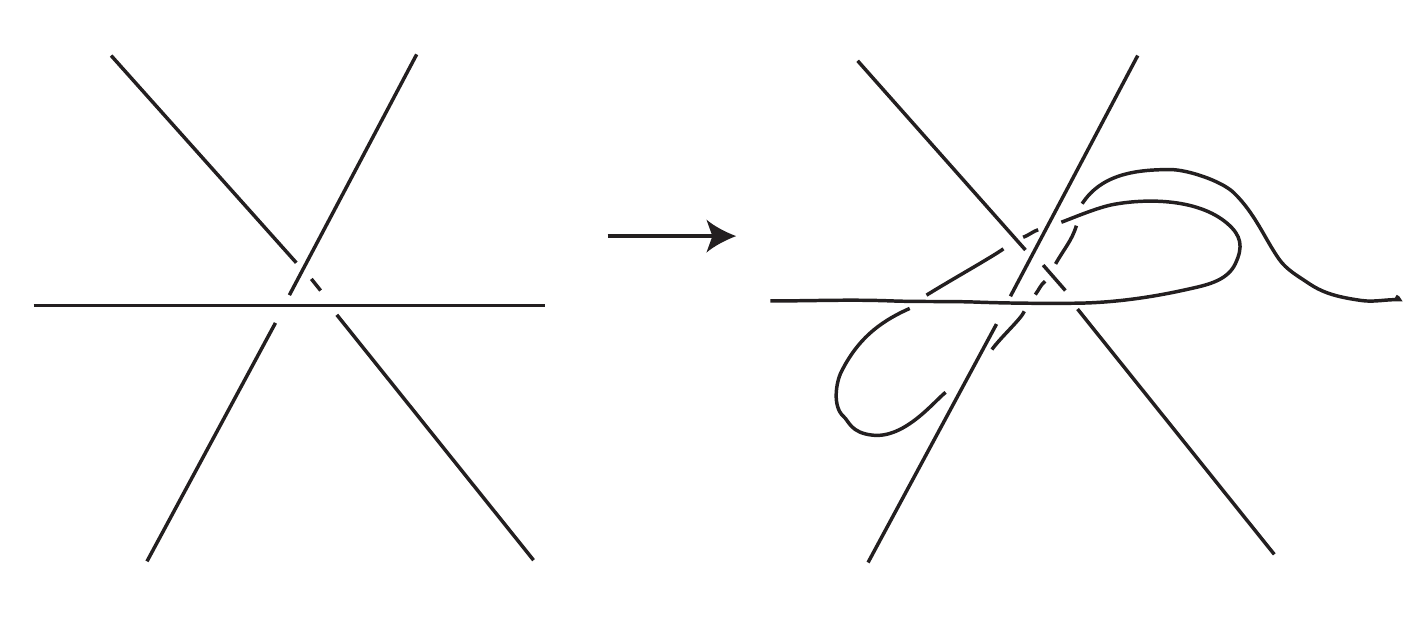}}
\caption{Turning a 3-crossing into a 5-crossing.}
\label{threetofive}
\end{figure}

The process just described of turning any $n$-crossing into an $(n+2)$-crossing immediately proves that, for any link $L$, 
\begin{align*}
c_2(L)&\ge c_4(L) \ge c_6(L)\ge \dots\\
c_3(L)&\ge c_5(L)\ge c_7(L)\ge \dots\\
\end{align*}
 
 However, the $CCC$ construction shows that  $c_2(L)>c_3(L)$ for every link $L$ because a set of $k$ 2-crossings lying on a crossing circle can be converted to a set of $k-1$ 3-crossings.
In this section, we prove that for any nontrivial nonsplit link $L$, other than the Hopf link, $c_3(L) > c_5(L)$. We begin with the following lemma.


\begin{lemma} \label{yesmonogons} If $D$ is a  connected 3-diagram with at least two 3-crossings and with at least one monogon among its complementary faces, then $D$ can be converted into a 5-diagram of the same link with fewer crossings.
\end{lemma}

\begin{proof} Note first that if we can convert any subset of the 3-crossings of $D$ into a smaller set of 5-crossings, then we can  convert each of the remaining 3-crossings into a  5-crossing as in Figure~\ref{threetofive} and hence obtain a 5-diagram of the same link with fewer crossings. Consider a complementary region, or face,  of $D$ that is a monogon and consider the 3-crossing $B$ incident to the monogon. Because $D$ is connected and has at least two 3-crossings, we are led to two possibilities, the first of which is shown in Figure~\ref{monogontrip}.  In this case we can eliminate the 3-crossing at $A$ by producing a 5-crossing at $B$ as shown in the figure. For example, suppose the two strands to be moved are the top and bottom  strands at $A$. We first pick up the top strand and lay it down on top of $B$. We then move what was the bottom strand at $A$ underneath the diagram, becoming the bottom strand at $B$. The other height arrangements at $A$ are handled similarly.

The other possibility leads us to the case shown in Figure~\ref{Hopf Link summand}. If the strand heights at crossing $B$ are not as shown in the figure (that is,  the``vertical" strand is not the middle strand) then the link is split, crossing $B$ can be eliminated, and the 5-crossing diagram  obtained by simply changing every remaining 3-crossing to a 5-crossing has fewer crossings. If, instead, the heights at crossing $B$ are as shown in the figure,  then the link has a Hopf link summand and the diagram can be changed as shown in the figure. But now we are again in the situation of Figure~\ref{monogontrip}. As before, this leads to a  5-diagram of the same link with fewer crossings.
\end{proof}

\begin{figure}[htbp]
\vspace*{13pt}
\centerline{\includegraphics*[scale=.7]{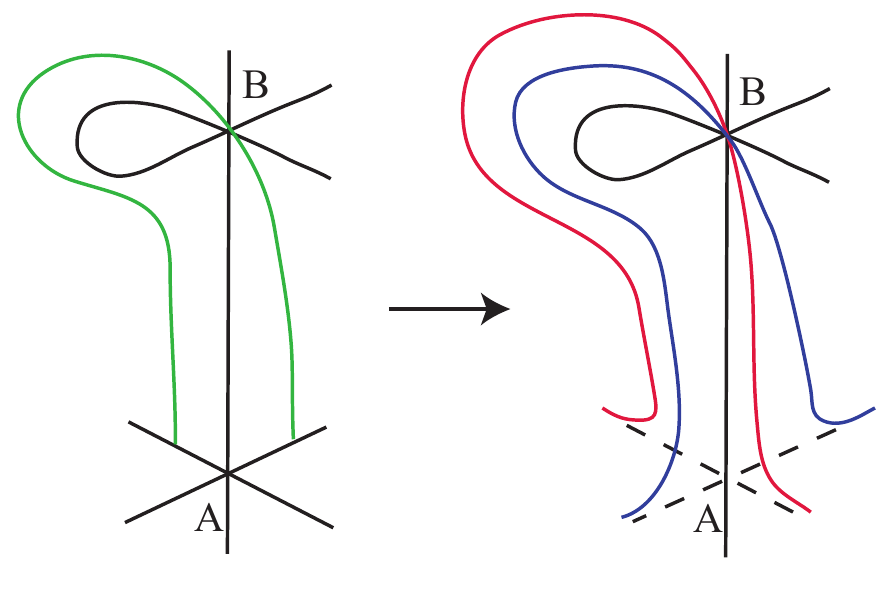}}
\caption{Using a monogon to replace two 3-crossings with one 5-crossing.}
\label{monogontrip}
\end{figure}

\begin{figure}[htbp]
\vspace*{13pt}
\centerline{\includegraphics*[scale=.5]{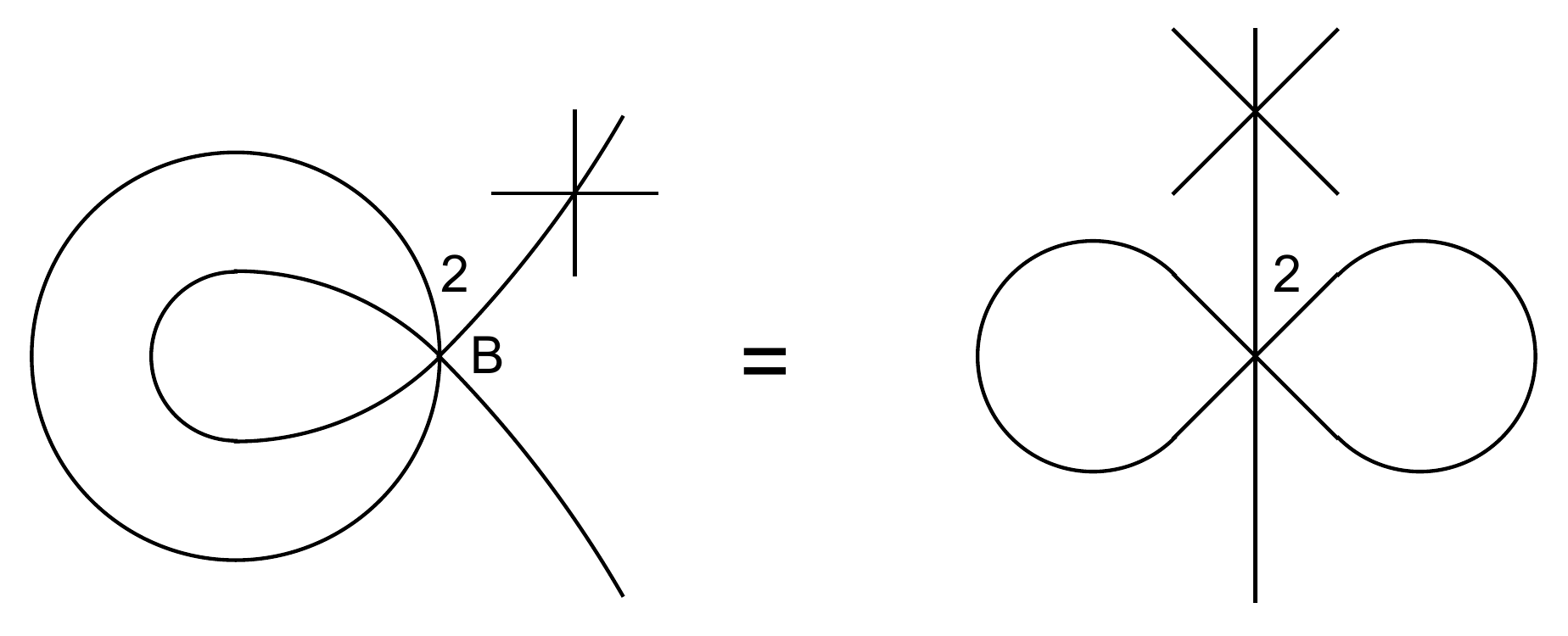}}
\caption{A monogon that gives a Hopf link summand.}
\label{Hopf Link summand}
\end{figure}

We now consider 3-diagrams with no monogons. If a complementary region of the diagram has a certain pattern of adjacent bigons on its boundary, then again, we will be able to find a 5-diagram with fewer crossings.

\begin{lemma} \label{lemma:badcases}Let $D$ be a connected 3-diagram of a link $L$. Suppose there exists a complementary region $F$ that is a polygon with $n>1$ edges, and the following holds:

\begin{enumerate}
\item If $n$ is even, at least every other edge of $F$ is shared with a bigon.
\item  If $n$ is odd, there are at least enough bigons on the boundary of $F$ such that, other than one pair of two adjacent edges, alternate edges are each shared with a bigon. 

\end{enumerate}
Then the 3-diagram can be converted into a 5-diagram with one less crossing than $D$. 
\end{lemma}

\begin{proof} See Figures \ref{badcaseseven} and \ref{badcasesodd} to see how we can convert the $n$ 3-crossings around the boundary of $F$ into $n-1$ 5-crossings by a process similar to that given in the proof of Lemma~\ref{yesmonogons}. Each of the remaining 3-crossings can then individually be turned into a 5-crossing.
\end{proof}

\begin{figure}[htbp]
\vspace*{13pt}
\centerline{\includegraphics*[scale=.6]{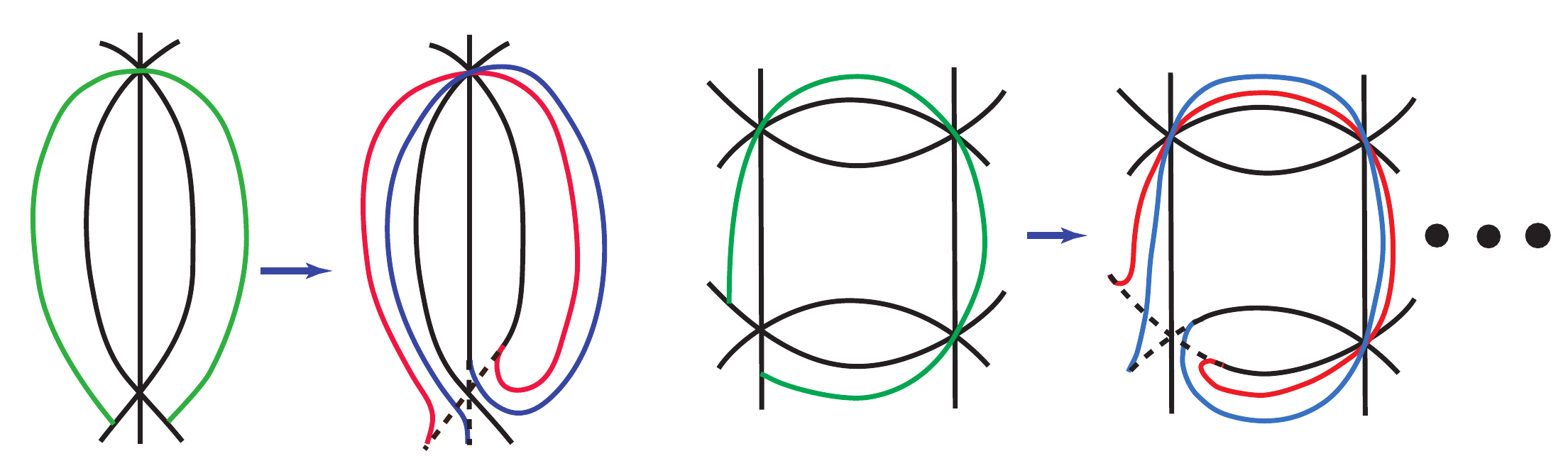}}
\caption{Converting 3-crossings into 5-crossings around an appropriate $n$-gon for $n$ even.}
\label{badcaseseven}
\end{figure}

\begin{figure}[htbp]
\vspace*{13pt}
\centerline{\includegraphics*[scale=.6]{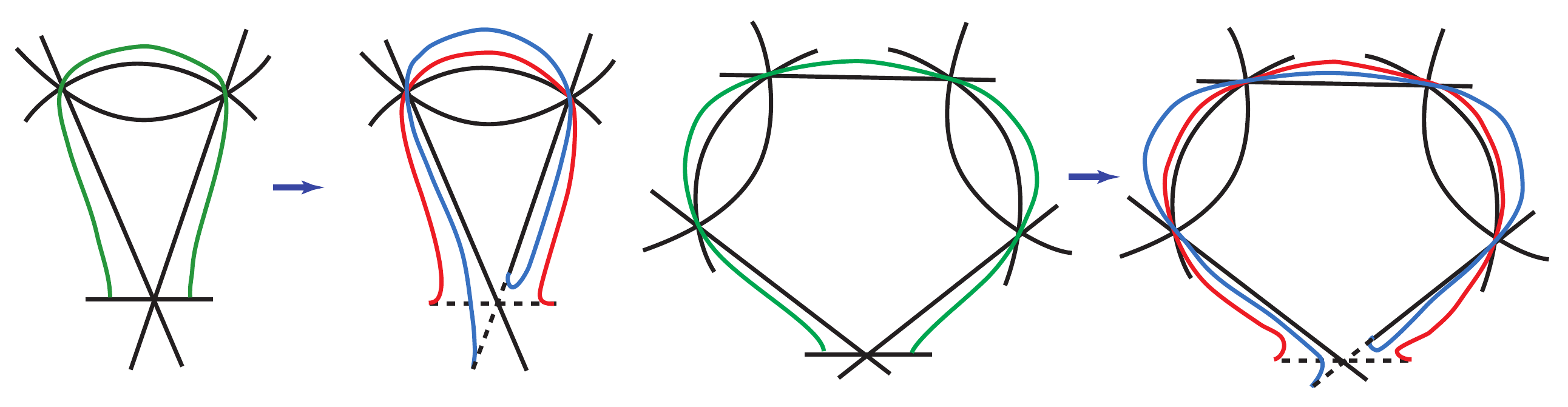}}
\caption{Converting 3-crossings into 5-crossings around an appropriate $n$-gon for $n$ odd.}
\label{badcasesodd}
\end{figure}

The goal now is to show that in any 3-diagram of a nontrivial nonsplit link that is not the Hopf link, a complemetary region which is either a monogon or a polygon of the kind described in Lemma~\ref{lemma:badcases} must exist. We will need the following observation.

\begin{lemma}Let $f_i$  be the number of faces with $i$ edges in any 3-diagram, including the outer region. Then
 $$2f_1 + f_2 = 6 + f_4 + 2f_5 + 3f_6 + \dots.$$
\end{lemma}

\begin{proof} Considering the  Euler characteristic gives $v-e+f = 2.$ But $e = 6v/2$, so $v = e/3$.
We now have $f = 2 + \frac{2e}{3}$, or equivalently $3f = 6 + 2e$,  where  $f = f_1 + f_2 + f_3 +\dots$ and $e = \frac{f_1 + 2f_2 + 3f_3 +\dots}{2}$. Thus, $3( f_1 + f_2 + f_3 + \dots )= 6 + f_1 +2f_2 + 3f_3 + \dots$ which yields the result. 
 \end{proof}

\begin{theorem} If $L$ is a nontrivial, nonsplit link, other than the Hopf link, then $c_3(L) > c_5(L).$
\end{theorem}

 \begin{proof} Let $D$ be a 3-diagram of $L$ that realizes the 3-crossing number $c_3(L)$. Because $L$ is nonsplit, $D$ must be connected. It is easy to check that the only connected 3-diagram with a single 3-crossing is either a trivial link or the Hopf link. Thus $D$ must have at least two crossings. If $D$ contains a monogon, then we are done by Lemma \ref{yesmonogons}.  Otherwise we have $f_1 = 0$ and
 \begin{equation}\label{formula for f_2}
f_2 = 6 + f_4 + 2f_5 + 3f_6 + \dots 
\end{equation}

If any of the cases that occur in Lemma \ref{lemma:badcases} appear in the 3-diagram $D$, we are done. So assume no such case occurs. In particular, this means that no two bigons share an edge and no bigon shares an edge with a triangle, etc.  
 Now let's count how many bigons could be present. Each  bigon can have no others on its boundary. Each triangle can have none on its boundary. Each quadrilateral can have at most two bigons on its boundary, where they are not opposite. Each pentagon can have at most two bigons on its boundary, for if it had three, two would be nonadjacent. More generally, for each $n$, if $n$ is even, there can be at most  $n-2$ bigons on the boundary to avoid one of the cases from Lemma \ref{lemma:badcases}. If $n$ is odd, there can be at most $n-3$ bigons on the boundary.

But counting bigons this way, we have counted them twice. So the conclusion is that \[f_2 \leq \frac{2f_4 + 2 f_5 + 4f_6 + 4f_7 + 6 f_8 + 6f_9 + \dots}{2}\]

So \[f_2 \leq  f_4 + f_5 + 2 f_6 + 2 f_7 + 3 f_8 + 3 f_9+ \dots\]

\medskip

But this is too small, contradicting Equation~\ref{formula for f_2}. Thus  $D$ must contain a face as in  Lemma \ref{lemma:badcases}, allowing us to convert $D$ into a 5-diagram with fewer crossings.
\end{proof}


\begin{thebibliography}{1}

\bibitem{Adams1}
Colin Adams.
\newblock Triple crossing number of knots and links.
\newblock {\em J. Knot Theory Ramifications}, 22(2):1350006, 2013.
\newblock  arXiv:math.GT/1207.7332

\bibitem{Adams2}
Colin Adams.
\newblock Quadruple crossing number of knots and links.
\newblock {\em Math. Proc. Cambridge Philos. Soc.}, 156(2):241--253, 2014.
\newblock arXiv:math.GT/1211.2726

\bibitem{Adams5}
Colin Adams, Orsola Capovilla-Searle, Jesse Freeman, Daniel Irvine, Samantha
  Petti, Daniel Vitek, Ashley Weber, and Sicong Zhang.
\newblock Multicrossing number for knots and the {K}auffman bracket polynomial.
\newblock {\em Math. Proc. of Cambridge Phil. Soc.}
\newblock to appear.
\newblock arXiv:math.GT/1407.4485

\bibitem{Adams3}
Colin Adams, Orsola Capovilla-Searle, Jesse Freeman, Daniel Irvine, Samantha
  Petti, Daniel Vitek, Ashley Weber, and Sicong Zhang.
\newblock Bounds on \"ubercrossing and petal numbers for knots.
\newblock {\em J. Knot Theory Ramifications}, 24(2):1550012, 2015.
\newblock arXiv:math.GT/1311.0526

\bibitem{Adams4}
Colin Adams, Thomas Crawford, Benjamin DeMeo, Michael Landry, Alex~Tong Lin,
  MurphyKate Montee, Seojung Park, Saraswathi Venkatesh, and Farrah Yhee.
\newblock Knot projections with a single multi-crossing.
\newblock {\em J. Knot Theory Ramifications}, 24(3):1550011, 2015.
\newblock arXiv:math.GT/1208.5742

\bibitem{AlexanderBriggs}
J.~W. Alexander and G.~B. Briggs.
\newblock On types of knotted curves.
\newblock {\em Ann. of Math. (2)}, 28(1-4):562--586, 1926/27.

\bibitem{Polyak}
Michael Polyak.
\newblock Minimal generating sets of {R}eidemeister moves.
\newblock {\em Quantum Topol.}, 1(4):399--411, 2010.
\newblock arXiv:math.GT/0908.3127

\bibitem{Reidemeister}
Kurt Reidemeister.
\newblock Knoten und {G}ruppen.
\newblock {\em Abh. Math. Sem. Univ. Hamburg}, 5(1):7--23, 1927.

\end{thebibliography}
 \end{document}